\newcommand{\N}{\mathbb{N}}
\newcommand{\R}{\mathbb{R}}

\newcommand{\Q}{\mathbb{Q}}

\newcommand{\rad}{\mathrm{rad}}
\newcommand{\cen}{\mathrm{cen}}
\newcommand{\radas}{\mathrm{rad}_{n\rightarrow \infty}}
\newcommand{\cenas}{\mathrm{cen}_{n\rightarrow \infty}}

\newcommand{\car}{{\mathbf{1}}}
\newcommand{\CAT}{{\mathrm{CAT}}}

\documentclass{conm-p-l}
\usepackage{amssymb,xcolor}
\usepackage{amsthm}

\newtheorem{theorem}{Theorem}[section]
\newtheorem{lemma}[theorem]{Lemma}
\newtheorem{corollary}[theorem]{Corollary}
\newtheorem{proposition}[theorem]{Proposition}

\theoremstyle{definition}
\newtheorem{definition}[theorem]{Definition}
\newtheorem{example}[theorem]{Example}

\theoremstyle{remark}
\newtheorem{remark}[theorem]{Remark}

\numberwithin{equation}{section}

\begin{document}

\title[On weak convergence in metric spaces]{On weak convergence in metric spaces}


\author{Giuseppe Devillanova}
\address{giuseppe.devillanova@poliba.it, Politecnico di Bari, via Amendola, 126/B, 70126 Bari, Italy}

\author{Sergio Solimini}
\address{sergio.solimini@poliba.it, Politecnico di Bari, via Amendola, 126/B, 70126 Bari, Italy}

\author{Cyril Tintarev}
\address{tintarev@math.uu.se, Uppsala University, box 480, 761 06 Uppsala, Sweden}


\date{today}

\begin{abstract}
This note gives an exposition of various extensions of the notion of weak convergence to metric spaces. They are motivated by applications, such as existence of fixed points of non-expansive maps, and analysis of the defect of compactness relative to gauge groups in Banach spaces, where weak convergence is generally less useful than respectively, asymptotic centers in [{\bf14}] and polar convergence in a preliminary version of [{\bf35}]). The note compares notions of convergence of weak type found in literature, in particular the notion of $\Delta$-convergence introduced by Lim in [{\bf25}], polar convergence introduced by the authors, and the modes of convergence of weak type introduced by Jost [{\bf20}], Sosov [{\bf36}] and Monod [{\bf28}] in Hadamard spaces. Some applications of polar convergence, such as the existence of fixed points for nonexpansive maps and a suitable variant of the Brezis-Lieb Lemma are produced.
\end{abstract}

\maketitle

\hfill {\em Dedicated to the memory of Teck-Cheong Lim}
\smallskip{}

\tableofcontents

\section{\label{intro}Introduction}
In a preliminary version of \cite{ST} a notion of convergence, similar to weak convergence, but based on the distance rather than on the continuity of linear forms, called polar convergence, has been presented as a natural mode of convergence which has been used, instead of weak convergence, in the analysis of the defect of compactness relative to gauge groups, allowing extension of the results, already known in the case of Hilbert spaces.
Later, the authors became aware that this notion is very close to the notion of $\Delta$-{\em convergence}, already introduced by T.C.~Lim in \cite{Lim} (a similar definition also appears in \cite{K}). The two definitions remain meaningful in metric spaces and, while they are different even in general linear spaces, they coincide under some regularity conditions, as for instance, Staples' rotundity (see \cite{Staples} and Definition~\ref{def:R} below) which, in the case of linear spaces, is equivalent to uniform convexity. 
In linear spaces the two definitions give a notion which is distinct from the usual weak convergence. Although they agree in Hilbert spaces (\cite{G-R-Book}), they are generically different even in many uniformly convex Banach spaces including $L^p$ spaces.

Opial, in the pioneering work \cite{Opial}, searched for conditions under which iteration sequences of non-expansive maps are weakly convergent, introducing the classical Opial condition (see \cite[Condition (2)]{Opial} and Definition~\ref{def:StrongOpial} below). As it soon became clear with the Edelstein's proof (\cite{Edelstein}) of the Browder fixed point theorem, fixed points are associated with $\Delta$-convergence of iterations (as asymptotic centers of iteration sequences), so that the implicit meaning of Opial condition, in broad terms, is to assert that weak and $\Delta$-convergence coincide. In particular, it is exactly the case in uniformly convex and uniformly smooth Banach spaces that weak and $\Delta$-convergence coincide if and only if the Opial condition holds. 

There are also definitions of weak convergence in metric spaces, which are based on the continuity of suitable test maps (essentially distance-minimizing projections onto compact convex sets) that play the role of the continuous linear forms on normed spaces. 
These notions are studied in detail for Hadamard spaces which are complete geodesic spaces which satisfy a suitable property which we shall call ``parallelogram inequality'' (see Definition \ref{de:parallineq} below).
It is well known that every linear normed Hadamard space is a Hilbert space, so it is not surprising that by analogy with Hilbert spaces, Hadamard spaces satisfy Opial condition, and that the polar convergence and the weak convergence in Hadamard spaces coincide. 
On the other side, it should be noted that in Banach spaces where Opial condition does not hold, in particular, in $L^p$ spaces, metric weak convergence defined via test maps still coincides with $\Delta$-convergence and not with the usual weak convergence. 
This indicates that distance-minimizing projections are generally nonlinear maps, and thus they cannot give rise to the usual weak convergence.

In essence, $\Delta$-convergence (with its close variations such as polar convergence) remains the only known mode of convergence of weak type, other than the usual weak convergence in normed vector spaces, and the rich literature on the subject gives a number of definitions of convergence that are equivalent, in the context they are stated, to $\Delta$-convergence, or are its close variants.  

Although this paper is focused on convergence of sequences, extension of its notions to nets is mostly trivial.
With the exception of Section \ref{se:polar topology} we do not address the problem of checking if the various modes of convergence we introduce are coming from a topology or not.
We have tried, despite of the brevity of the paper, to make a self-contained exposition, including even some proofs of properties readily available in the literature.

The paper is organized as follows.
In Section \ref{se:convergences} we introduce polar and $\Delta$-convergence, which  
are directly depending on the distance. 
In Section \ref{Rotund} we introduce metric spaces which are (uniformly) Staples' rotund. These spaces are a metric counterpart of uniformly convex normed spaces. Similarly to uniformly convex Banach spaces, where weak and weak-star topology agree and some local compactness properties  (such as the Banach-Alaoglu compactness theorem) hold, in complete uniformly rotund metric spaces we prove that $\Delta$-convergence and polar convergence agree and enjoy some compactness properties (see Theorem \ref{thm:polarcompactness}).

In Section \ref{se:polar topology} we discuss the question posed by Dhompongsa,  Kirk and Panyanak \cite{DKP}, whether
$\Delta$-convergence (or polar convergence) can be associated with a topology, 
and we give explicit examples of metric spaces in which polar convergence is not derived from a topology.

In Section \ref{se:polarBS} we study polar convergence in Banach spaces pointing out, by examples, differences between polar convergence and weak convergence. In particular, when a Banach space is uniformly convex and uniformly smooth, $\Delta$-convergence can be characterized as weak convergence in the dual (see Theorem \ref{thm:conj}).

In Section \ref{se:spaces} we recall Opial condition and introduce ``parallelogram inequality'' on metric spaces which gives an equivalent definition of $\CAT(0)$ spaces (see \cite[Definition 2.9]{EspFern}) or spaces of Alexandrov nonpositive curvature (see \cite[Definition 2.3.1 and the subsequent remark]{JostBook}), as well as to deduce their Staples' rotundity. 
Minimal distance projections are employed in Section \ref{se:WCTM} where we extend to metric spaces the notion of weak convergence by replacing the linear forms with distance-minimizing projections onto geodesic segments.
We point out that on general normed spaces continuous linear forms are not distance-minimizing projections onto straight lines, so while this definition is appropriate for Hadamard spaces, which generalize Hilbert spaces, distance-minimizing projections do not extend the notion of weak convergence to general metric spaces.
Moreover in Theorem~\ref{te:Equivalence} we prove the equivalence of weak and polar convergence.
Before ending the section we also recall two topologies $\mathcal T_c$ and $\mathcal T_w$ defined on Hadamard spaces in \cite{Monod}
(while searching for a topology associated with $\Delta$-convergence (or weak convergence)), and we discuss the equivalence of the topology $\mathcal T_c$ with the weak topology. Moreover we show that in Hilbert spaces $T_w$ coincides with the topology induced by the norm.

In Section \ref{se:app} we apply polar ($\Delta$) convergence to the fixed points theory. In particular we introduce 
a condition (polar asymptotical regularity condition, see Definition \ref{de:PAR}) for a nonexpansive map at a point $x$  which guarantees the polar convergence of the iterations sequence at this point. 

Finally we discuss a version of the Brezis-Lieb Lemma \cite{Brezis-Lieb} where, see Theorem~\ref{thm:newbl}, the assumption of $a.e.$ convergence is replaced by the assumption of ``double weak'' convergence (namely polar convergence and usual weak convergence to the same point, which are both implied by $a.e.$ convergence of a bounded sequence in $L^p$ for $1<p<+\infty$). 
Moreover, given that there is a stronger version of Brezis-Lieb Lemma in Hilbert spaces, based on weak convergence only, we get an analogous result for Hadamard spaces.
\section{\label{se:convergences}Weak convergences of polar type}

\begin{definition}\textbf{($\Delta$-limit)}\label{de:DeltaConv}
Let $(E,d)$ be a metric space. A sequence $(x_{n})_{n\in\mathbb{N}}\subset E$ is said to $\Delta$-converge to a point $x\in E$ (see \cite{Lim}), and we shall write $x_{n}\stackrel{\Delta}{\rightarrow} x$, if
\begin{equation}
\label{eq:Dconvergence}
\limsup_n d(x_{k_n},x)\leq \limsup_n d(x_{k_n},y)
\end{equation}
for any subsequence $(x_{k_n})_{n\in\N}$ of $(x_{n})_{n\in\N}$ and for every $y\in E$.
\end{definition}

The following proposition, whose proof is straightforward, gives a characterization of a $\Delta$-limit.

\begin{proposition}
\label{pr:Dconvergence}
Let $(E,d)$ be a metric space. A sequence $(x_{n})_{n\in\mathbb{N}}\subset E$ $\Delta$-converges to a point $x\in E$ if and only if 
\begin{equation}
\label{eq:DconvChar}
\forall y\in E \; : \; d(x_n,x)\leq d(x_n,y) +o(1).
\end{equation}
\end{proposition}
Here and in what follows the Edmund Landau symbol $o(1)$ denotes a sequence of real numbers convergent to zero.

\begin{definition}\textbf{(Strong $\Delta$-limit)}\label{de:delta}
Let $(E,d)$ be a metric space. A sequence $(x_{n})_{n\in\mathbb{N}}\subset E$ is said to 
strongly-$\Delta$ converge to a point $x\in E$ (see \cite{Lim}), and we shall write $x_{n}\stackrel{s-\Delta}{\rightarrow} x$, if
\begin{equation}
\label{eq:Dconvergencen}
(\exists \lim_n d(x_n,x) \quad \mbox{ and}) \quad \forall y\in E\; : \; \lim_n d(x_n,x)\leq \liminf_n d(x_n,y)\;.
\end{equation}
\end{definition}
Note that if a sequence  $(x_n)_{n\in\N}$ has no bounded subsequence then it strong${-\Delta}$ converges to every point. On the contrary if $(x_n)_{n\in\N}$ is a bounded sequence, obviously, if $x_{n}\stackrel{s-\Delta}{\rightarrow} x$ then  $x_{n}\stackrel{\Delta}{\rightarrow} x$.
The notions of $\Delta$ and strong-$\Delta$ convergence of a given sequence $(x_n)_{n\in\N}$ can be further clarified by means of the notion of 
{\em asymptotic center} (denoted by $\cenas x_n$) and {\em asymptotic radius} (denoted by $\radas x_n$) (definitions of asymptotic radius and asymptotic centers can be found for instance in \cite{Edelstein} where Edelstein gives a proof of Browder fixed point theorem \cite{Browder} based 
on these notions. In this connection see page 251 of the paper \cite{Reich-remark} and page 18 in \cite{G-R-Book}). 
We emphasize that, while the asymptotic radius always exists and is uniquely determined, asymptotic centers may not exist or may be not uniquely determined. Therefore, the symbol $\cenas x_n$ must be understood in the same sense as the limit symbol in a topological space which is not assumed to be Hausdorff.
Here we use, for the sake of convenience, an elementary reformulation of the original definition of asymptotic centers
and asymptotic radius of a sequence $(x_n)_{n\in\N}$ as minimum points and, respectively,
infimum value of the following functional on $E$
\begin{equation}
\label{eq:Ias}
I_{as}(y)=\limsup_{n}d(x_{n},y).
\end{equation}

\begin{remark}
\label{re:DconvergenceCenter}
Let $(E,d)$ be a metric space. A sequence $(x_{n})_{n\in\mathbb{N}}\subset E$ $\Delta$-converges to a point $x\in E$ if and only if $x$ is an asymptotic center of every subsequence. On the other hand 
the strong-$\Delta$ convergence of $(x_n)_{n\in\N}$ to $x$ means that every subsequence has the same asymptotic radius (equal to $I_{as}(x)$).
This last property allows to prove that any asymptotic center of the whole sequence $(x_n)_{n\in\N}$ is an asymptotic center of every subsequence and therefore it is a $\Delta$-limit of $(x_n)_{n\in\N}$. (Indeed if 
$(x_{k_n})_{n\in\N}$ is a subsequence of $(x_n)_{n\in\N}$, since, by strong-$\Delta$ convergence, 
$\radas x_{k_n}=\radas x_n$, the inequality $\limsup_n d(x_{k_n},x)\leq \limsup_n d(x_n,x)=\radas x_n$ forces $x$ to be an asymptotic center of $(x_{k_n})_{n\in\N}$). Obviously the converse implication is not true.
\end{remark}

\begin{remark}
\label{re:ordering}
In other terms, let $\Xi$ be the set of bounded sequences of elements in $E$, 
for every $\xi$, $\zeta\in \Xi$, we shall write $\xi\leq \zeta$ if the sequence $\zeta$ is extracted from 
$\xi$ after a finite number of terms (note that, in spite of the notation, 
$\leq$ is not an ordering, since it is not antisymmetric).
The function $f$ which maps every sequence 
$\xi=(x_n)_{n\in\N}\in \Xi$ into $f(\xi)=-\radas x_n$ is a ``increasing'' function 
(i.e. if $\xi\leq \zeta$ then $f(\xi)\leq f(\zeta)$).
We can pass to a coarser relation $\preceq$, which is a true ordering and makes $f$ strictly increasing, 
by setting, for any $\xi,\zeta\in\Xi$, 
$\xi\preceq\zeta$ if $\xi=\zeta$ or if $\xi\leq \zeta$ and $f(\xi)<f(\zeta)$.
Under this notation, we can reformulate the second part of the previous remark by stating that 
$x_{n}\stackrel{s-\Delta}{\rightarrow} x$ if and only if $(x_n)_{n\in\N}$ is maximal for $\preceq$ and $x$ is an asymptotic center of $(x_n)_{n\in\N}$.

Note that the ordered set $(\Xi, \preceq)$ is countably inductive in the sense specified in \cite[Appendix A]{MadSol}.
Indeed, if $(\xi_n)_{n\in\N}\subset \Xi$ is increasing with respect to $\preceq$, 
after trowing away a finite number of terms from each sequence $\xi_n$
(in order to make each $\xi_{n+1}$ a subsequence of $\xi_n$) and passing to a diagonal selection, one obtains an upper bound $\xi$ of the whole sequence.
\end{remark}

\begin{remark}
\label{re:PolarVsDelta}
Note that a (strong) $\Delta$-limit is not necessarily unique, even in the case of a bounded sequence. For instance, if $(A_n)_{n\in\N}$ is a decreasing sequence of measurable sets of $\R$ such that, for any $n\in\N$, $A_n\setminus A_{n+1}$ is if positive measure, by setting, for any $n\in\N$, $x_n=\car_{A_{n+1}}-\car_{A_n\setminus A_{n+1}}$ we get a sequence of bounded functions which (since for any $n\neq m\in\N$ 
$\|x_n-x_m\|_\infty=2$) does not admit any subsequence with asymptotic radius strictly smaller than $1$.
Then, given $\bar{n}\in\N$, any function $x$ such that 
$\|x\|_{\infty}\leq 1$ and such that 
$x_{|A_{\bar{n}}}=0$ (for instance $x=\car_{\R\setminus A_{\bar{n}}}$) satisfies $\|x_n-x\|_\infty=1$
for $n>\bar{n}$. Therefore $x$ is an asymptotic center of the sequence $(x_n)_{n\in\N}$ and $\radas x_n=1$. Since, as already proved, $(x_n)_{n\in\N}$ does not admit any subsequence with asymptotic radius strictly smaller than $1$, it follows that $\radas x_{k_n}=1$ for any subsequence $(x_{k_n})_{n\in\N}$. Therefore, by Remark 
\ref{re:DconvergenceCenter}, $x$ is a strong-$\Delta$ limit of the sequence $(x_n)_{n\in\N}$.
\end{remark}

\begin{definition}\textbf{(Polar limit)}\label{de:pol} 
Let $(x_{n})_{n\in\mathbb{N}}$
be a sequence in a metric space $(E,d)$. One says that $x\in  E$ 
is a \emph{polar limit} of $(x_{n})_{n\in\mathbb{N}}$ and we shall write $x_n\rightharpoondown x$, if for every
$y\neq x$ there exists $M(y)\in\mathbb{N}$ such that 
\begin{equation}
d(x_{n},x)<d(x_{n},y)\text{ for all }n\ge M(y).\label{eq:mwl}
\end{equation} 
\end{definition}

\begin{remark}
\label{re:PimplyD}
It is immediate from comparison between (\ref{eq:DconvChar}) and (\ref{eq:mwl}) that polar convergence implies 
$\Delta$-convergence. Moreover (\ref{eq:mwl}) guarantees the uniqueness of polar limit. On the other hand, one can deduce from Remark~\ref{re:PolarVsDelta} that $\Delta$-convergence and polar convergence generally do not coincide.
\end{remark}
\begin{remark}
\label{re:polarVsDelta}
The following properties for $\Delta$, strong-$\Delta$ and polar limits of sequences in a metric space $(E,d)$ are immediate:
\begin{itemize}
\item [{(i)}] If $x_n\stackrel{s-\Delta}{\to}x$ or $x_n\rightharpoondown x$, then $x_n\stackrel{\Delta}{\to}x$.
\item [{(ii)}] If $x_n\to x$ (in metric), then $x_n\stackrel{s-\Delta}{\to}x$ and $x_n\rightharpoondown x$.
\item [{(iii)}] If $(x_n)_{n\in\N}$ is $\Delta$ (resp. strong-$\Delta$, resp. polarly)-convergent to a point $x$, then any subsequence $(x_{k_n})_{n\in\N}$ of $(x_n)_{n\in\N}$ is $\Delta$ (resp. strong-$\Delta$, resp. polarly)-convergent to $x$.
\item [{(iv)}] $(x_n)_{n\in\N}$ is $\Delta$ (resp. polarly)-convergent to a point $x$ if and only if any subsequence 
$(x_{k_n})_{n\in\N}$ of $(x_n)_{n\in\N}$ admits a subsequence which is $\Delta$ (resp. polarly)-convergent to 
$x$.
\item [{(v)}] If $(x_n)_{n\in\N}$ is $\Delta$-convergent to a point $x$, then it admits a subsequence which is strong-$\Delta$ convergent to 
$x$.
\item [{(vi)}] If $E$ is a Banach space, then the sequence $((-1)^n x)_{n\in\N}$, with any $x\neq 0$, has no $\Delta$-limit. \end{itemize}
\end{remark}
It follows from (iv) and (v) combined that a sequence $(x_{n})_{n\in\mathbb{N}}$ $\Delta$-converges to a point $x$ if every subsequence admits a subsequence strong-$\Delta$-convergent to $x$.

\begin{proposition}
\label{prop:Cauchy}
Let $(E,d)$ be a metric space, $x\in E$ and let $(x_{n})_{n\in\mathbb{N}}$ be a precompact sequence (or, in particular, a Cauchy sequence) in $E$. 
If $x_{n}\stackrel{\Delta}{\rightarrow} x$,
then
$x_{n}\to x$.
\end{proposition}
\begin{proof}
Since the sequence $(x_{n})_{n\in\mathbb{N}}$ is precompact, for any $\varepsilon>0$, there exists a finite $\varepsilon$-
net $\mathcal{N}_{\varepsilon}$ of $(x_{n})_{n\in\mathbb{N}}$. Then,
by (\ref{eq:DconvChar}), for large $n$, we have 
$d(x_{n},x)\leq \min_{y\in\mathcal{N_{\varepsilon}}}d(x_{n},y)+o(1)<\varepsilon$,
i.e. $d(x_{n},x)\to 0$.
\end{proof}

\begin{definition}
\label{de:Deltacompleteness} (\cite{Lim}) A metric space $(E,d)$ is called $\Delta$-complete (or is said to satisfy the $\Delta$-completeness property) if every bounded sequence admits an asymptotic center.
\end{definition}
An easy maximality argument, see for instance \cite[Theorem A.1]{MadSol}, yields the following result.

\begin{theorem}
\label{thm:compactness} (\cite[Theorem 3]{Lim}) 
Let $(E,d)$ be a $\Delta$-complete metric space. Then every bounded sequence in $E$ has a strong-$\Delta$-convergent subsequence.
\end{theorem}

\begin{proof}
Since the ordered set $(\Xi,\preceq)$ introduced in Remark \ref{re:ordering} is countably inductive and since the function 
$f$ which maps every sequence 
$\xi=(x_n)_{n\in\N}\in \Xi$ into $f(\xi)=-\radas x_n$ is a real strictly increasing function
(see Remark \ref{re:ordering}), by using \cite[Theorem A.1]{MadSol} one obtains that every sequence 
$(x_n)_{n\in\N}\in \Xi$ has a maximal subsequence for $\preceq$.
Since $E$ is $\Delta$-complete, this subsequence has an asymptotic center to which, since it is maximal, it is strongly-$\Delta$ convergent (see Remark \ref{re:ordering}).
\end{proof}

This argument has been employed by Lim in \cite{Lim} and, in a very close setting, in the proof of 
\cite[Lemma 15.2]{GK} (although the lemma is set in the Banach space, the proof extends to the metric space verbatim, as it has been already observed in \cite{KirkPanayak}). Note that the existence of a strictly increasing real valued function $f$ and the separability of $\R$ also allow to prove that the countable inductivity leads to inductivity and so one can deduce the existence of a maximal element by Zorn Lemma as in \cite[Proposition 1]{Lim}. However, the direct argument in \cite[Theorem A.1]{MadSol} looks even simpler than the proof of the inductivity.

There is a number of publications where $\Delta$ convergence or strong-$\Delta$ convergence (not always under that name) is applied to problems related to fixed points in Hadamard spaces - to mention just few, \cite{Bacak, DKS, KirkPanayak}, but we first consider here a larger class of metric spaces, namely {\em uniformly rotund spaces} introduced by John Staples \cite{Staples}. 

\section{\label{Rotund}Rotund metric spaces}

\begin{definition}
\label{def:R}
A metric space $(E,d)$ is a (uniform) SR (``Staples
rotund'') metric space (or satisfies (uniformly) property SR) if
there exists a function $\delta:\;(\mathbb{R}_{+})^{2}\to\mathbb{R}_{+}$
such that for any $r,\,\overline{d}>0$, set $\delta=\delta(r,\overline{d})$, for any $x,y\in E$ with 
$d(x,y)\ge\overline{d}$:
\begin{equation}
\label{eq:SR}
\tag{SR}
\mathrm{rad}(B_{r+\delta}(x)\cap B_{r+\delta}(y))\leq r-\delta.
\end{equation}
\end{definition}
\noindent For the reader's convenience we recall that the Chebyshev
radius of a set $X$ in a metric space $(E,d)$ is the infimum of
the radii of the balls containing $X$. In other words $\mathrm{rad}(X)=\inf_{x\in E}\sup_{y\in X} d(x,y)$.
Moreover, the Chebyshev radius and the Chebyshev centers of a set $X\subset E$ can also be defined, analogously to the asymptotic radius and asymptotic centers of a sequence, by replacing the functional $I_{as}$ in (\ref{eq:Ias}) by 
\begin{equation}
\label{eq:I_X}
I_X: x\in E\mapsto \sup_{y\in X} d(x,y).
\end{equation}

\begin{remark}
If $\delta$ is continuous, one can replace, in the above definition,
the occurrences of $r+\delta$ by $r$. Moreover, from Definition \ref{def:R} it is immediate that uniformly convex normed vector spaces (see \cite[Definition 1.e.1]{L-T}) are uniformly SR metric spaces. Furthermore, any uniformly SR normed vector space is uniformly convex. To show this, let $\|x\|\leq 1$, $\|y\|\leq 1$ such that $\|x-y\|\ge\varepsilon$. It follows easily that both $0$ and $x+y$ belong to $\bar B_1(x)\cap \bar B_1(y)$.
By rotundity, set
$\delta=\delta(1,\varepsilon)$, we have $\mathrm{rad}(\bar B_1(x)\cap \bar B_1(y))\leq 1-\delta$. Therefore,  
$\|x+y-0\|\le 2\mathrm{rad}(\bar B_1(x)\cap \bar B_1(y))\le 2-2\delta$ and so $\|\frac{x+y}{2} \|\le 1-\delta$ follows, thus proving uniform convexity. 
\end{remark}

\begin{remark}
Staples' rotundity is close to the uniform ball convexity in Foertsch \cite{Foertsch} and to the uniform convexity for hyperbolic metric spaces given by Reich and Shafrir \cite{RS}. The latter properties assume existence of a midpoint map which is not assumed by Staples, although, on the other hand, Staples' definition makes an additional assumption that roughly speaking amounts to the continuity of the modulus of convexity with respect to radius of the considered balls. In Banach spaces the above definitions coincide with the notion of uniform convexity. Particular examples of SR metric spaces are given by Hadamard spaces and, more generally, by $\mathrm{CAT}(0)$ spaces which will be discussed in Section \ref{se:spaces} below. 
\end{remark}

For complete Staples rotund metric spaces, 
\cite[Theorem 2.5]{Staples} and \cite[Theorem 3.3]{Staples} state, respectively, uniqueness and  existence of the asymptotic center of any bounded sequence (giving actually the $\Delta$-completeness of such spaces, see Definition \ref{de:Deltacompleteness}).
We are going to prove the just mentioned properties as one of the following claims which hold true in complete SR metric spaces.

For complete SR metric spaces we have the following results.
\begin{itemize}
\item[{\bf a})] $\Delta$-convergence and polar convergence coincide.
\item[{\bf b})] Every bounded sequence has a unique asymptotic center.
\item[{\bf c})] The space is $\Delta$-complete.
Therefore (the sequential compactness property in) Theorem \ref{thm:compactness} also holds for polar convergence as stated in the following theorem.
\end{itemize}

\begin{theorem}
\label{thm:polarcompactness}
Let $(E,d)$ be a complete SR metric space. Then every bounded sequence in $E$ has a polarly convergent subsequence.
\end{theorem}

Actually claim {\bf c}) is just a restatement of claim {\bf b}) and therefore it does not need to be proved.
\begin{lemma}
\label{lem:strong-dom} Let $(E,d)$ be a SR metric space. Let $(x_{n})_{n\in\mathbb{N}}\subset E$
be a bounded sequence and let $x\in E$ be such that (\ref{eq:DconvChar}) holds true.
Then, for each element $z\in E$, $z\ne x,$ there exist positive
constants $n_{0}$ and $c$ depending on $z$ such that 
\begin{equation}
\label{eq:NewStrongDom}
d(x_{n},x)\le d(x_{n},z)-c\mbox{ \quad for all\quad\ \ensuremath{n\ge n_{0}}},
\end{equation}
\noindent and so $(x_n)_{n\in\N}$ polarly converges to $x$.
\end{lemma}
\begin{proof}
If the assertion is false, by (\ref{eq:DconvChar}), we can find $z\neq x$ and a subsequence $(x_{k_{n}})_{n\in\mathbb{N}}$
such that $d(x_{k_{n}},x)-d(x_{k_{n}},z)\to0$. Passing again to a
subsequence we can also assume that $d(x_{k_{n}},x)\to r>0$. Set
$\overline{d}=d(x,z)$ and $\delta=\delta(r,\overline{d})$.
Since, for large $n$, $x_{k_{n}}\in B_{r+\delta}(x)\cap B_{r+\delta}(z)$,
we can deduce from (\ref{eq:SR}) the existence of $y\in E$ such
that  $d(x_{k_{n}},y)<r-\delta$, in contradiction to (\ref{eq:DconvChar}).
\end{proof}
\begin{corollary}
\label{co:strong-dom}
Let $(E,d)$ be a SR metric space. Let $(x_{n})_{n\in\mathbb{N}}\subset E$
be a bounded sequence and let $x\in E$ be such that (\ref{eq:DconvChar}) holds true. Then, for any compact set $\mathcal{K}\subset E$ such that $x\notin \mathcal{K}$ there exists $\overline{n}=\overline{n}(\mathcal{K})$ such that for any $n\geq \overline{n}$ and for any $z\in \mathcal{K}$ (\ref{eq:NewStrongDom}) holds true.
\end{corollary}
 By taking into account the characterization of $\Delta$-limit given by Proposition \ref{pr:Dconvergence}, 
Lemma \ref{lem:strong-dom} guarantees that if $x_n \stackrel{\Delta}{\rightarrow} x$, then $x_n\rightharpoondown x$. This proves claim {\bf a)}, since, as pointed out in Remark \ref{re:PimplyD}, polar convergence always implies $\Delta$-convergence to the same point.

\begin{lemma}
Let $(E,d)$ be a SR metric space. Let $I_{as}$ be defined by (\ref{eq:Ias})
relative to a given sequence $(x_{n})_{n\in\mathbb{N}}\subset E$. Then, for any
$\overline{d}>0$, there exists $\varepsilon>0$ such that if $x,y\in E$
satisfy $\max(I_{as}(x),I_{as}(y))<\inf I_{as}+\varepsilon$, then $d(x,y)<\overline{d}$.
\end{lemma}
\begin{proof}
If $\inf I_{as}=0$ one can take $\varepsilon=\frac{\overline{d}}{2}$,
otherwise, since $E$ is a SR metric space (see (\ref{eq:SR})), one can take $\varepsilon=\delta(\inf I_{as},\overline{d})$. 
\end{proof}
Two remarkable and immediate consequences of the above lemma trivially
follow.
\begin{corollary}
\label{co:uni+Cau}
Let $(E,d)$ be a SR metric space and let $I_{as}$ be as above. Then
\end{corollary}

\begin{itemize}
\item [{(i)}] the functional $I_{as}$ admits at most one minimum point;
\item [{(ii)}] any minimizing sequence of the functional $I_{as}$ is a Cauchy sequence. 
\end{itemize}

\begin{corollary}
\label{cor:unicity}Let $(E,d)$ be a SR complete metric space, then
any bounded sequence $(x_{n})_{n\in\mathbb{N}}\subset E$,
admits a unique asymptotic center. Moreover any
minimizing sequence for $I_{as}$ converges to the asymptotic center
of the sequence $(x_{n})_{n\in\mathbb{N}}$. 
\end{corollary}
In other words, Corollary \ref{cor:unicity} states in particular
that in a SR complete metric space every bounded sequence $(x_{n})_{n\in\mathbb{N}}$
has a unique asymptotic center $\cenas x_{n}$, so that claim {\bf b)} is proved. 

\section{Polar neighborhoods}
\label{se:polar topology}

In this section we discuss the question, posed by Dhompongsa, Kirk and Panyanak in \cite{DKP}, if $\Delta$-convergence (or polar convergence) can be associated with a topology.

Let $(E,d)$ be a metric space, let $Y\subset E$, $x\notin Y$, and
let 
\begin{equation}
\mathcal{N}_{Y}(x)=\bigcap_{y\in Y}\mathcal{N}_{y}(x)=\left\{ z\in E\,|\, d(z,x)<d(z,y),\forall y\in Y\right\} ,\label{eq:pol-neigh}
\end{equation}
where
\begin{equation}
\label{eq:Nyx}
\mathcal{N}_{y}(x)=\left\{ z\in E\,|\, d(z,x)<d(z,y)\right\} .
\end{equation}
In other words, the set $\mathcal{N}_{Y}(x)$ is the set
of all points in $E$ which are strictly closer to $x$ than to $Y$.
\begin{remark}
\label{rem4.1}A trivial restatement of Definition \ref{de:pol} is
that, if $x\in E$ and $(x_{n})_{n\in\N}\subset E$ , then $x_{n}\rightharpoondown x$
if and only if, for any finite set $Y\not\ni x$, there exists $M(Y)\in\N$
such that $x_{n}\in\mathcal{N}_{Y}(x)$ for all $n\ge M(Y)$.\end{remark}
\begin{definition}
\label{de:p-nbd}Let $x\in E$. We shall call \emph{polar neighborhoods
}(briefly p-nbd) of the point $x$ all the subsets $V\subset E$ containing
a set $\mathcal{N}_{Y}(x)$ given by (\ref{eq:pol-neigh}), where
$Y\subset E$ is any finite set such that $x\notin Y$.\end{definition}
\begin{remark}
Polar convergence can be finally tested, as it follows from Remark
\ref{rem4.1}, by using polar neighborhoods. Indeed, if $x\in E$
and $(x_{n})_{n\in\N}\subset E$, then $x_{n}\rightharpoondown x$
if and only if $x_{n}\in V$ for large $n$ for any polar neighborhood
$V$ of $x$.
\end{remark}
One can easily see that the set of polar neighborhoods of a given
point $x$ is a filter of parts of $E$. It is well known (see \cite[Prop. 1.2.2 ]{Bourbaki})
that the union of such filters is a neighborhood base of a (unique)
topology $\mathcal{T}$ if and only if the following further property
is satisfied
\begin{equation}
\label{eq:TN}
\forall x\in E,\forall V\;\mbox{p-nbd. of }x,\exists U\;\mbox{p-nbd of }x\;\mbox{s.t. }\forall z\in U:\; V\;\mbox{is a p-nbd of }z.
\end{equation}
By Definition~\ref{de:p-nbd} the above condition can be more explicitly
stated as 
\begin{equation}
\forall x,y\in E,\, x\neq y,\exists Y\;\mbox{finite set s.t. }\forall z\in\mathcal{N}_{Y}(x)\exists Z\;\mbox{s.t. }\mathcal{N}_{Z}(z)\subset\mathcal{N}_{y}(x).
\label{eq:BTN}
\end{equation}
When (\ref{eq:TN}) (or equivalently (\ref{eq:BTN})) is satisfied we shall call $\mathcal{T}$ \emph{polar
topology} induced by $d$ on $E$. Since, for any $x,y\in E$ if $2r=d(x,y)$,
$B_{r}(x)\subset\mathcal{N}_{y}(x)$, polar topology is in general
a coarser topology of that usually induced by $d$ by the classical
notion of neighborhood. We shall refer to the latter topology as
\emph{strong topology} induced by $d$ when we want to distinguish
it from the polar one.

We shall show now that (\ref{eq:TN}) is true
in some cases and false in others, and even when (\ref{eq:TN}) is true, $\mathcal{T}$
not always coincides with the strong topology induced by $d$.
\begin{example}
\label{Ex:discrete}Let $\delta$ be the discrete metric on $E$.
Then, setting $d=\delta$ in (\ref{eq:Nyx}), it is easy to see that $\mathcal{N}_{y}(x)=\left\{ x\right\}$
for all $x,y\in E$, $x\neq y$. Therefore the polar topology induced
by $\delta$ is the discrete topology as well as the strong topology.
\end{example}
In Hilbert spaces polar topology coincides with weak topology, and
thus is different from the strong topology in the infinite-dimensional
case.
\begin{example}
\label{Ex:H}
Let $E$ be a Hilbert space. Then, for any $x,y\in E$,
$x\neq y$, we have that $z\in\mathcal{N}_{y}(x)$ means $\|z-x\|^{2}-\|z-y\|^{2}<0$,
namely $(2z-(x+y))\cdot(y-x)<0$. If, with an invertible change of
variables, we set $x+y=2a$, $y-x=v$, we get that $\mathcal{N}_{y}(x)=\left\{ z\in E\,|\,(z-a)\cdot v<0\right\} $.
So $\mathcal{N}_{y}(x)$ gives, for $x,y\in E$, $x\neq y$, a base
of the weak topology.
\end{example}
Combining the previous two examples we can discuss the following case.
\begin{example}
\label{Ex:H-1}
Let $E$ be a finite dimensional vector space and let
$d$ be defined as $d(x,y)=|x-y|+\delta(x,y)$ where $\delta$ is
as in Example \ref{Ex:discrete}. Then it is easy to check that the
polar convergence induced by $d$ agrees with the natural convergence of
vectors while the strong topology induced by $d$ is the discrete
topology. In other words the strong topology induced by $d$ is the discrete
topology (which, see Example \ref{Ex:discrete}, coincides with both the strong and the polar topology
induced by $\delta$, i.e. by neglecting the contribution of $|x-y|$), 
while the polar topology induced by $d$
is the natural topology (which, see Example \ref{Ex:H}, coincides with both the strong and
the polar topology induced by $|x-y|$, i.e. by neglecting
the contribution of $\delta$). 
\end{example}
Finally we show an example in which (\ref{eq:TN}) is not satisfied
and therefore polar topology does not exist.
\begin{example}
Let $E=\R$ and let $D$ be the Dirichlet function (i.e. $D(x)=1$
for $x\in\Q$ and $D(x)=0$ for $x\notin\Q$ ). Let $d$ be defined
as $d(x,y)=|x-y|+(1+D(x-y))\delta(x,y)$ with $\delta$ as in Example
\ref{Ex:discrete}. It is easy to see that $d$ is a metric (indeed,
if $y\notin\left\{ x,z\right\} $ then $d(x,z)\leq|x-z|+2\leq|x-y|+1+|y-z|+1\leq d(x,y)+d(y,z)$).
We shall prove that (\ref{eq:BTN}) does not hold.
Fix any $\overline{x}$, $\overline{y}\in\R$, $\overline{x}\neq\overline{y}$,
$|\overline{x}-\overline{y}|\leq1$ and $\overline{y}-\overline{x}\notin\Q$.
Consider $\mathcal{N}_{Y}(\overline{x})$ with an arbitrary finite
set $Y\not\ni\bar{x}$ and let $\varepsilon>0$ such that $2\varepsilon<|\overline{x}-y|$
for all $y\in Y$. Then fix $\overline{z}\in\R$ such that $|\overline{x}-\overline{z}|<\varepsilon$
and $\overline{x}-\overline{z}\notin\Q$. This choice implies that
\begin{equation}
\forall y\in Y:\quad2|\overline{x}-\overline{z}|<|\overline{x}-y|\quad\mbox{and so}\quad|\overline{x}-\overline{z}|<|\overline{z}-y|\label{eq:x-z}
\end{equation}
 by the triangle inequality. Therefore 
\begin{equation}
d(\overline{x},\overline{z})=|\overline{x}-\overline{z}|+1<|\overline{z}-y|+1\leq d(\overline{z},y)\quad\forall y\in Y,\label{eq:x-z-2}
\end{equation}
proving that $\overline{z}\in\mathcal{N}_{Y}(\overline{x})$. In the
same way, let $Z$ be a given finite set such that $\overline{z}\notin Z$
and fix $\eta>0$ such that $2\eta<|\overline{z}-z|$ $\forall z\in Z$.
Fix $v\in\R$ such that $|\overline{z}-v|<\eta$ and $\overline{x}-v\in\Q$
(therefore $v-\overline{y}\notin\Q,\;|v-\overline{z}|\notin\Q$).
With analogous estimates to (\ref{eq:x-z}) and (\ref{eq:x-z-2})
we see that $v\in\mathcal{N}_{Z}(\overline{z})$. Finally, since $v-\overline{y}\notin\Q$,
$|\overline{x}-\overline{y}|\leq1$ and $\overline{x}-v\in\Q$, 
\[
d(v,\overline{y})=|v-\overline{y}|+1\leq|v-\overline{x}|+|\overline{x}-\overline{y}|+1\leq|v-\overline{x}|+2=d(v,\overline{x}).
\]
Hence for any $\mathcal{N}_{Y}(\overline{x})$ we can find $\overline{z}\in\mathcal{N}_{Y}(\overline{x})$
such that every neighborhood $\mathcal{N}_{Z}(\overline{z})\not\subset\mathcal{N}_{\overline{y}}(\overline{x})$
and so (\ref{eq:BTN}) does not hold.
\end{example}

\section{Polar convergence in Banach spaces}
\label{se:polarBS}

In a Hilbert space the polar limit of a sequence is also its weak
limit and vice versa, see Example \ref{Ex:H}. The original
argument, although brought up to prove a weaker statement, is due
to Opial, see \cite{Opial} where the following condition is also formulated see  \cite[Condition (2)]{Opial} and \cite{ST} for more details. We can interpret this condition, in reflexive spaces, as an equivalence between weak and polar convergence.

\begin{definition}[Opial condition]
\label{def:StrongOpial} 
One says that a Banach space $(E, \|\;\|)$ satisfies
Opial condition if for every sequence $(x_{n})_{n\in\N}\subset E$, which
is weakly convergent to a point $x\in E$, 
\begin{equation}
\liminf_n\|x_{n}-x\|<\liminf_n\|x_{n}-y\|\mbox{ for every }y\in E, y\neq x.
\label{eq:OpCon}
\end{equation}
\end{definition}
Note that (\ref{eq:OpCon}) with the $<$ sign replaced by $\leq$ remains equivalent in a SR metric space, (the equivalence follows by repeating the argument in the proof of Lemma \ref{lem:strong-dom}).

It is worth mentioning that van Dulst in \cite{VD} proved that a separable uniformly convex
Banach space can be provided with an equivalent norm satisfying the Opial
condition. 

Opial also gives an example, see \cite[Section 5]{Opial}, which we can interpret
in terms of polar convergence, of a bounded sequence in $L^{p}((0,2\pi))$,
$p\neq2$, $1<p<\infty$, whose polar limit and weak limit do not coincide. 
We shall come back soon with more powerful tools to the case of $L^p$ but we first present
two examples, suggested by Michael Cwikel, which illustrate some differences
between polar and weak convergence in a Banach space. They respectively show that a weakly convergent
sequence in $\ell^1$ does not need to be bounded and that a weakly converging sequence in $\ell^\infty$ does not need to have a polar limit.
\begin{example}
\label{ex:C1}
Let $E$ be the sequence space $\ell^{1}$, and let, for any $k\in\mathbb{N}$, 
$x_{k}:=k(\delta_{k,n})_{n\in\mathbb{N}}$
(we have used the Kronecker delta values, i.e. $\delta_{k,n}=1$ for $k=n$ and $\delta_{k,n}=0$
otherwise). Since, for each fixed sequence $\alpha=(\alpha_{n})_{n\in\mathbb{N}}\in\ell^{1}$,
we have $\lim_{k\to\infty}\left\Vert x_{k}-\alpha\right\Vert _{\ell^{1}}-\left\Vert x_{k}\right\Vert _{\ell^{1}}=\left\Vert \alpha\right\Vert _{\ell^{1}}$ we get that the sequence $(x_{k})_{k\in\mathbb{N}}$ is polarly convergent
to the zero element of $\ell^{1}$ . 
\end{example}
\begin{example}
\label{ex:C2}
Let $E$ be the sequence space $\ell^{\infty}$ and
let $x_{k}:=(\delta_{k,n})_{n\in\mathbb{N}}$ for each $k\in\mathbb{N}$.
Clearly the sequence $(x_{k})_{k\in\mathbb{N}}$ converges weakly
to the zero of $\ell^{\infty}$. However, zero is not a polar limit
of the sequence, since for the sequence $\alpha=(1,1,1,\dots)$ we
have $1=d(0,x_{k})=d(\alpha,x_{k})$ for all $k$. Let $\beta:=(\beta_{n})_{n\in\mathbb{N}}$
be an arbitrary nonzero element of $\ell^{\infty}$. Then $\beta_{n_{0}}\ne0$
for some integer $n_{0}$ and we can define the sequence $\gamma:=(\gamma_{n})_{n\in\mathbb{N}}$
by setting $\gamma_{n_{0}}=0$ and $\gamma_{n}=\beta_{n}$ for all
$n\ne n_{0}$. We see that $d(\beta,x_{k})\ge d(\gamma,x_{k})$
for all $k>n_{0}$ which shows that $\beta$ cannot be the polar limit
of $(x_{k})_{k\in\mathbb{N}}$. 
\end{example}

Even if most of the assertions in the following statements  hold under weaker assumptions as well, we shall work in this section in uniformly convex and uniformly smooth Banach spaces i.e. in a Banach space $E$ such that both $E$ and $E^\prime$ satisfy SR (see \cite[p. 32]{J-L}).
This implies, in particular, that $E$ is reflexive.
A duality relation which holds in  this setting is given in \cite{ST}.
We shall use in this section a notion of duality map which to every $x$ belonging to a uniformly convex and uniformly smooth Banach space $E$ associates the unique element $x^\prime$ in the dual space $E^\prime$ such that
\begin{equation}
\label{eq:duality}
\langle x, x^\prime\rangle =\|x\|^2=\|x^\prime\|^2.
\end{equation}
Note that  $x^\prime= \|x\| x^\ast$, where $x^\ast$ is as 
in \cite{C} or in \cite{R}.
In other terms, $x^\ast$ is the Fr\'ech\`et derivative of the norm $\|x\|$ while $x^\prime$ is the Fr\'ech\`et derivative of $\frac{1}{2} \|x\|^2$, $x^\prime=x$ when $E$ is a Hilbert space.
Here and in what follows weak convergence of a sequence $(x_n)_{n\in\N}$ in a Banach space to a point $x$ is denoted as $x_{n}\rightharpoonup x$. We shall need the following technical lemma.
\begin{lemma}
\label{le:Tlemma}
Let $(E,\|\,\|)$ be a Banach space.
Let $f:\R_+\rightarrow \R_+$ be a function such that $\lim_{t \rightarrow +\infty} f(t)=0$. Let $(\xi_n)_{n\in\N}\subset E^\prime$ be such that
for any $z\in E$, $\|z\|=1$, and for all $s>0$ there exists $\overline{n}=\overline{n}(s,z)$ such that 
for $n\geq \overline{n}$
\begin{equation}
\label{eq:xprimeineq}
|\langle\xi_n,z\rangle|\leq f(s\|\xi_n\|) \|\xi_n\|.
\end{equation}
Then $\xi_n$ weakly$^\ast$ converges to $0$.
\end{lemma}
\begin{proof}
It suffices to prove the assertion for a subsequence, therefore we can assume without restrictions that there exists $C_1>0$ such that $\|\xi_n\|\geq C_1$ for any $n$.
By taking $s=1$ in (\ref{eq:xprimeineq}) we deduce that 
$\left(f(\|\xi_n\|) \|\xi_n\|\right)^{-1}\xi_n$ is pointwise bounded and by the Uniform Boundedness Principle (see \cite[Theorem 2.2]{B-book}) it is bounded. Therefore, by the assumption on $f$, we deduce that exists $C_2>0$ such that $\|\xi_n\|\leq C_2$. Then, by applying (\ref{eq:xprimeineq}) we deduce $|\langle\xi_n,z\rangle|\leq C_2 f(s\|\xi_n\|)$, and so the assertion follows, 
since $s \|\xi_n\|\geq s C_1$, by letting $s\rightarrow +\infty$. 
\end{proof}
The following result, which shows the duality relation between the weak and the polar convergence in Banach spaces, follows from \cite[Theorem 3.8]{ST}, which is 
formulated in an equivalent way in terms of $(x_n^\ast)_{n\in\N}$, combined with \cite[Theorem 3.5]{ST}.

\begin{theorem}[Duality principle]
\label{thm:conj}
Let $(E,\|\,\|)$ be a uniformly convex and uniformly smooth Banach space. 
Let $(x_n)_{n\in\N}\subset E$ be any given 
sequence.
Then $x_{n}\rightharpoondown 0$ in $E$
if and only if $x_{n}^\prime\rightharpoonup 0$ in $E^\prime$. 
\end{theorem}

\begin{proof}
Assume $x_n\rightharpoondown 0$. Note that for any $z\in E$ and for any $n\in\N$ we can find a unitary coefficient $\alpha_n=\alpha_n(z)$ such that 
\begin{equation}
\label{eq:alpha_n}
|\langle x_n^\prime,z\rangle| =-\alpha_n \langle x_n^\prime,z\rangle,
\end{equation}
then, given $\varepsilon>0$, since 
the scalar coefficient 
$\varepsilon \alpha_n\in \partial B_{\varepsilon}(0)$,
by Corollary \ref{co:strong-dom}, 
with $\mathcal{K}= (\partial B_{\varepsilon}(0)) z$, 
there exists $\overline{n}=\overline{n}(\varepsilon,z)$ such that for any $n\geq \overline{n}$,
\begin{equation}
\label{eq:polineq}
0 \leq  \|x_n+\varepsilon \alpha_n(z) z\|^2-\|x_n\|^2.
\end{equation}
Since $\frac{1}{2}\|x\|^2$ is uniformly Fr\'ech\`et differentiable on the normalized elements of $E$
(see \cite[pag. 61]{L-T}), 
by combining (\ref{eq:polineq}) with (\ref{eq:alpha_n}), we have for large $n$
\begin{equation}
\label{eq:ineq}
0 \leq  
\|x_n\|^2 
\left( \left\|\frac{x_n}{\|x_n\|}+ \frac{\varepsilon \alpha_n }{\|x_n\|}\right\|^2-1
\right)
\leq
-2\varepsilon 
|\langle x_n^\prime, z\rangle| + \|x_n\|^2 r\left(\frac{\varepsilon}{\|x_n\|}\right), 
\end{equation}
where $r$ is a function from $\R_+$ into $\R_+$ such that 
$\frac {r(s)}{s}\rightarrow 0$ as $s\rightarrow 0$.
Then, for large $n$,
\begin{equation}
\label{eq:ineq}
|\langle x_n^\prime, z\rangle | \leq 
\left(\frac{\|x_n\|}{2 \varepsilon} r\left(\frac{\varepsilon}{\|x_n\|}\right)\right) \|x_n\|,
\end{equation}
so, since $\|x_n\|=\|x_n^\prime\|$, by applying Lemma \ref{le:Tlemma}, with
$f(t)=2^{-1} t \, r(t^{-1})$, we deduce 
$x_n^\prime \rightharpoonup 0$.
Conversely, let $x_n^\prime \rightharpoonup 0$, then for any $z$ we get
\begin{equation}
\label{eq:angle}
\|x_n\|^2= \langle x_n^\prime, x_n\rangle= \langle x_n^\prime ,x_n-z \rangle + \langle x_n^\prime, z\rangle
\leq 
\|x_n\| \, \|x_n-z\|+o(1),
\end{equation}
so the thesis follows from Lemma \ref{lem:strong-dom}.
\end{proof}
Note that the spaces $E$ and $E^\prime$ in the duality principle above are interchangeable, and therefore weak convergence and polar convergence can be also interchanged in the formulation of Theorem~\ref{thm:conj}.
Opial's examples in \cite{Opial} of sequences in $L^p$ spaces, $1<p<\infty$, whose weak and polar limit do not coincide, can be easily explained in the light of Theorem~\ref{thm:conj}: one takes normalized sequences of oscillating functions $(f_n)_{n\in\N}$ which weakly converge to zero but have positive and negative oscillations of different heights. Taking into account that the dual elements $f_n^\prime$ in $L^p$ are given by the formula
$f^\prime=\|f\|^{2-p} |f|^{p-2}f$, if $p\neq 2$, 
we easily deduce that the positive and negative oscillations of $f_n^\prime$ are no longer balanced,  $f_n^\prime\not\rightharpoonup 0$ and so, by Theorem  \ref{thm:conj} $f_n\not\rightharpoondown 0$.

This situation is completely different if we know that $(f_n)_{n\in\N}$ is bounded and converges pointwise $a.e.$ to $0$. 
In such a case, by weak compactness, it also weakly converges to $0$ and the dual sequence $(f_n^\prime)_{n\in\N}$ still converges $a.e.$ to $0$ and therefore weakly.
So, by Theorem \ref{thm:conj}, $f_n\rightharpoondown 0$.
In other terms we have the following statement.
\begin{remark}
\label{re:ae->w&p}
Bounded $a.e.$ converging sequences in $L^p$ are both weakly and polarly convergent to their $a.e.$ limits.
\end{remark}

This is always the case in $\ell^p$ spaces, with $1<p<+\infty$, where weak convergence implies pointwise convergence.
So, for $1<p<+\infty$, the spaces $\ell^p$ satisfy Opial condition, which is therefore not a characteristic condition of Hilbert spaces. This does not extend to the cases $p=1$ and $p=+\infty$ discussed in Examples \ref{ex:C1} and \ref{ex:C2}.
Finally, the following result, in contrast to the general situation considered in Example \ref{ex:C1}, was proved in \cite[Theorem 3.5]{ST} and is in some sense implicit in Lemma \ref{le:Tlemma}.

\begin{proposition}
\label{pr:Boundedness}
If $(E,\|\,\|)$ is a 
uniformly convex and
uniformly smooth Banach space, then
any polarly converging sequence in $E$ is bounded.
\end{proposition}

\begin{proof}
If $x_n \rightharpoondown x$, then by Theorem \ref{thm:conj}, $(x_n-x)^\prime_{n\in\N}$ is bounded in $E^\prime$. By definition, this implies that $(x_n-x)_{n\in\N}$ is bounded in $E$.
\end{proof}

\section{\label{se:spaces}Hadamard spaces}

\begin{definition}
\label{de:parallineq}
Let $(E,d)$ be a metric space, we shall say that $(E,d)$ satisfies the ``parallelogram inequality'' if given any
points $x, y\in E$, there exists a point $m\in E$ such that for every point $z\in E$,
\begin{equation}
\label{cat}
\tag{PI}
d^2(z, m) \leq 
\frac{1}{2}	\left(d^2(z, x) 
+ d^2(z, y)\right)
-\frac{1}{4}d^2(x, y).
\end{equation}
\end{definition}

\begin{definition}
\label{de:CAT0}
We shall say that a metric space $(E,d)$ is a $\mathrm{CAT}(0)$ space (see \cite[Definition 2.9]{EspFern}) or is a space of Alexandrov nonpositive curvature (see 
\cite[Definition 2.3.1 and the subsequent remark]{JostBook}) if it satisfies inequality (\ref{cat}) and it is geodesic. 
\end{definition}

\begin{remark}
\label{re:PI->SR}
Note that a metric space $(E,d)$ which satisfies inequality (\ref{cat}) is SR.
\end{remark}

In order to clarify why we are calling (\ref{cat}) a ``parallelogram inequality'' let us first observe that, when $x$ and $y$ have a middle point $c$ (i.e. $\exists c\in E$ such that $d(x,c)=d(y,c)=\frac{1}{2}d(x,y)$) then $c$ is the point $m$ whose existence is asked in (\ref{cat}).

Given any points $x,y\in E$ if they have a middle point then
\begin{equation}
\label{eq:d=2r}
d(x, y)= 2\rad(\{x,y\}).
\end{equation}
Moreover, spaces which satisfy inequality (\ref{cat}), have the following property.
\begin{proposition}
\label{pr:middle}
Let $(E,d)$ be a metric space which satisfies inequality (\ref{cat}). Any pair of points $x,y\in E$ admits a (unique)
middle point $c$ if and only if (\ref{eq:d=2r}) holds.
In such cases $c=\cen(\{x,y\})$ and $m=c$ in $(\ref{cat})$.
\end{proposition}

\begin{proof}
Fix $x,y \in E$ and set $X=\{x,y\}$. Let $(c_n)_{n\in\N}$ be a minimizing sequence of the functional $I_X$ defined in (\ref{eq:I_X}). Then, by replacing $z$ by $c_n$ in (\ref{cat}) we deduce that $d(c_n,m)\rightarrow 0$ and so $m=\cen(\{x,y\})$.
The center is unique since (by Remark \ref{re:PI->SR}) $E$ is SR.
\end{proof}

Note that on a linear normed space, by setting $z^\prime=2m-z$, i.e. a reflection of the point $z$ about point $m$, the points $x,y,z$ and $z^\prime$ are vertices of a parallelogram and so 
(\ref{cat}) becomes
\begin{equation}
\label{eq:PIneq}
d^2(z,z^\prime)+
d^2(x,y)
\leq
2\left(d^2(z,x)+ d^2(z,y)\right)
\end{equation}
and expresses an inequality involving the squares of the diagonals and of the four sides of the parallelogram and justifies the name of (PI).
Moreover, since we can associate to any parallelogram a dual one whose sides are the diagonals of the previous one, equality must hold in (\ref{eq:PIneq}). Therefore the space can be equipped with a scalar product, i.e. linear complete spaces which satisfy (\ref{cat}) are Hilbert spaces.

\begin{remark}
\label{re:middle}
Note that if a metric space $(E,d)$ satisfies (\ref{eq:d=2r}) for all $x,y$ and satisfies inequality (\ref{cat}), then by recursively applying Proposition \ref{pr:middle}, we deduce that for any pair of points $x, y\in E$ and for any dyadic number $\alpha\in [0,1]$ there exists a (unique) point $c_\alpha\in E$ such that 
\begin{equation}
\label{eq:c_alpha}
d(c_\alpha,y)=\alpha d(x,y)\quad \mbox{  and } \quad d(c_\alpha,x)=(1-\alpha) d(x,y).
\end{equation}

We can also deduce, by recursively applying (\ref{cat}), that 
for any $z\in E$
\begin{equation}
\label{catalpha}
d^2(z, c_\alpha)  
\leq 
(1-\alpha) d^2(z, x) + \alpha d^2(z, y)- 
d(c_\alpha,x) d(c_\alpha,y)
\end{equation}
or, equivalently,
\begin{equation}
\label{catalphas}
d^2(z, c_\alpha)  
\leq 
(1-\alpha) d^2(z, x) + \alpha d^2(z, y)- 
\alpha(1-\alpha)d^2(x, y).
\end{equation}
\end{remark}

When $(E,d)$ is in addition a complete metric space, 
$c_\alpha$ is defined by continuity for any $\alpha\in [0,1]$ and gives rise to a (unique) geodesic, contained in $E$, which joins $x$ to $y$.
Therefore any complete metric space $(E,d)$ which satisfies inequality 
(\ref{cat}) and (\ref{eq:d=2r}) for all $x,y$ is a (uniquely) geodesic space.
Finally (\ref{catalphas}) also holds for any $\alpha\in [0,1]$.

\begin{definition}
\label{de:Hadamard}
A metric space $(E,d)$ is an Hadamard space if it is a complete $\mathrm{CAT}(0)$ space.
\end{definition}

\begin{remark}
\label{re:Hilbert}
It is easy to conclude, given our observations above about the parallelogram inequality, that linear Hadamard spaces are Hilbert spaces.
\end{remark}

As a rule of thumb one may consider Hadamard spaces as a ``metric generalization'' of Hilbert spaces, and expect that Hadamard spaces have some properties typical of Hilbert spaces but not found in all Banach spaces. 

A few general properties of Hadamard spaces follow.

\begin{definition}[convex sets]
\label{de:convex}
Let $(E,d)$ be a geodesic metric space and $C\subset E$.
We shall say that $C$ is convex if  any two points $x,y\in C$ can be connected by a geodesic contained in $C$.
\end{definition} 
Note that, when $(E,d)$ is a geodesic $\mathrm{CAT}(0)$ (resp. Hadamard) space and $C\subset E$ is a complete (resp. closed) convex set, then $C$ satisfies (\ref{cat}).
Indeed, the middle point between any two points $x,y$ belongs to the geodesic joining them 
which is contained in $C$.
Therefore, $C$ satisfies (\ref{eq:SR}) and so for any $x\in E$, any minimizing sequence of elements in $C$ of the functional $I_{\{x\}}$ (defined by (\ref{eq:I_X}) for $X=\{x\}$) is a Cauchy sequence (see Corollary \ref{co:uni+Cau}). 
Therefore, since $C$ is complete,
this sequence converges to the unique point $p_C(x)$ of $C$ which minimizes the distance from $x$.
It is easy to prove that $p_C:E \rightarrow E$ is a projector (i.e. that $p_C\circ p_C=p_C$). It is called the distance-minimizing projector onto $C$. 

\begin{lemma}[Pythagoras Inequality]
\label{le:PythIneq}
Let $(E,d)$ be a $\mathrm{CAT}(0)$ space.
Let $\Gamma\subset E$ be a geodesic,
then for all $x\in E$ and for all $y\in \Gamma$
\begin{equation}
\label{eq:PythIneq}
d^2(x,y)\geq d^2(x,p_\Gamma(x))+d^2(p_\Gamma(x),y).
\end{equation}
\end{lemma}

\begin{proof}
For any $\alpha\in [0,1]$, by suitably applying (\ref{catalpha})
with 
$c_\alpha$ as in (\ref{eq:c_alpha}),
one gets that
$d^2(x,c_\alpha)\leq (1-\alpha)d^2(x,p_\Gamma(x))+\alpha d^2(x,y)-\alpha(1-\alpha) d^2(p_\Gamma(x),y)$ and, since $c_\alpha\in \Gamma$ and therefore $d(x,p_\Gamma(x))\leq d(x,c_\alpha)$, we deduce 
$d^2(x,y)\geq d^2(x,p_\Gamma(x))+(1-\alpha)d^2(p_\Gamma(x),y)$ the thesis follows by taking $\alpha$ as close to $0$ as we want.
\end{proof}

The above property easily extends to any compact convex set.

\begin{corollary}
\label{co:PythIneqGen}
Let $(E,d)$ be a $\mathrm{CAT}(0)$ space.
Let $C\subset E$ be a compact convex set, then for all $x\in E$ and for all $y\in C$
\begin{equation}
\label{eq:PythIneqC}
d^2(x,y)\geq d^2(x,p_C(x))+d^2(p_C(x),y).
\end{equation}
\end{corollary}

\begin{lemma}
\label{le:nonexpansive}
Let $(E,d)$ be a $\mathrm{CAT}(0)$ space.
Let $C\subset E$ be a complete convex set, then the function $p_C:E\rightarrow E$ is a nonexpansive map, i.e.
\begin{equation}
\label{eq:nonexpansive}
d(p_C(x),p_C(y))\leq d(x,y)
\end{equation}
holds for all $x,y\in E$.
\end{lemma}

\begin{proof}
By applying (\ref{eq:PythIneqC}) we deduce the two symmetric inequalities
\begin{equation}
\label{eq:Pythagora1}
d^2(x,p_C(y))\geq d^2(x,p_C(x))+ d^2(p_C(x),p_C(y))
\end{equation}
\begin{equation}
\label{eq:Pythagora2}
d^2(y,p_C(x))\geq d^2(y,p_C(y))+ d^2(p_C(y),p_C(x)).
\end{equation}
Denoting by $m$ the middle point between $p_C(x)$ and $y$ we get from (\ref{cat}) and 
(\ref{eq:Pythagora2})
\begin{equation}
\label{eq:sdiag1}
d^2(x,m) 
\leq \frac{1}{2}d^2(x,p_C(x))+ \frac{1}{2}d^2(x,y)-\frac{1}{4}d^2(y,p_C(y))-\frac{1}{4}d^2(p_C(y),p_C(x))
\end{equation}
\begin{equation}
\label{eq:sdiag2}
d^2(p_C(y),m)
\leq \frac{1}{4}d^2(y,p_C(y))+\frac{1}{4}d^2(p_C(y),p_C(x)).
\end{equation}
Then, by the Schwartz inequality and by adding (\ref{eq:sdiag1}) to (\ref{eq:sdiag2}), we obtain 
\begin{equation}
\label{eq:Holder^2}
d^2(x,p_C(y))
\leq 2(d^2(x,m)+d^2(m,p_C(y)))\\
\leq 
d^2(x,p_C(x))+ d^2(x,y).
\end{equation}
From (\ref{eq:Pythagora1}) combined with (\ref{eq:Holder^2})
the assertion follows.
\end{proof}

\section{Definitions of convergence by via test maps}
\label{se:WCTM}

We consider the following definitions which extend to metric spaces the notion of weak convergence on Hilbert spaces by substituting the linear forms by distance-minimizing projections onto geodesic segments.
In the case of normed spaces this does not necessarily define the usual weak convergence, since distance-minimizing projections onto  straight lines do not generally coincide with continuous linear forms.
\begin{definition}
\label{de:Jost}
Let $(E,d)$ be a Hadamard space. A point $x\in E$ is called the weak limit of a sequence $(x_n)_{n\in\N}\subset E$ if for every
geodesic $\Gamma$ starting at $x$, $(p_\Gamma(x_n))_{n\in\N}$ converges to $x$. In this case, we say that $(x_n)_{n\in\N}$ weakly converges to $x$ and we shall write 
$x_n\rightharpoonup x$.
\end{definition}

As far as we know the definition is due to J. Jost (see \cite[Definition 2.7]{Jost}) for sequences in geodesic $\mathrm{CAT}(0)$ spaces.
Subsequently in \cite{Sosov} Sosov has introduced essentially the same definition under the name of $\varphi$-convergence, and a variant of it called $\psi$ convergence (which utilizes projections onto geodesics trough $x$ rather that geodesics starting from $x$).
As we show below in Theorem~\ref{te:Equivalence}, in Hadamard spaces weak convergence can be equivalently defined as follows.
\begin{definition}
\label{def:weak}
Let $(E,d)$ be a Hadamard space. 
A point $x\in E$ is called the weak limit of a sequence $(x_n)_{n\in\N}\subset E$ if for every compact convex set $C\subset E$ containing $x$, the sequence $(p_C(x_n))_{n\in\N}$ converges to $x$.
\end{definition}
Finally we can also drop the requirement of convexity in the definition of weak convergence if we use a multivalued projection, which allows to bring up the following definition which makes sense in every metric space.
\begin{definition}
\label{def:polartest}
Let $(E,d)$ be a metric space. 
A point $x\in E$ is called the weak limit of a sequence $(x_n)_{n\in\N}\subset E$ if for every compact set $C\subset E$ containing $x$ and for all sequences $(x^\prime_n)_{n\in\N}\subset C$ of minimal distance from $(x_n)_{n\in\N}$, the sequence $(x^\prime_n)_{n\in\N}$ converges to $x$.
\end{definition}
By analogy with Hilbert spaces it is natural that Hadamard spaces may satisfy Opial condition (see Definition \ref{def:StrongOpial}). We have the following result (see \cite{EspFern}, \cite{Bacak1}).
\begin{theorem}
\label{te:Equivalence}
Let $(E,d)$ be a Hadamard space, let $x\in E$ and $(x_n)_{n\in\N}\subset E$ be a sequence. Then the following statements are equivalent:
\begin{itemize}
\item[a)] $(x_n)_{n\in\N}$ weakly converges to $x$ in the sense of Definition 
\ref{def:polartest};
\item[b)] $(x_n)_{n\in\N}$ weakly converges to $x$ in the sense of Definition \ref{def:weak};
\item[c)] $(x_n)_{n\in\N}$ weakly converges to $x$ in the sense of Definition \ref{de:Jost}.
\item[d)] $(x_n)_{n\in\N}$ is polarly convergent to $x$.
\end{itemize}
\end{theorem}

\begin{proof}
a)$\Rightarrow$ b) 
is trivial and b) $\Rightarrow$ c) follows since any geodesic is a compact convex set.

c)$\Rightarrow$ d)
Fix $y\in E$, $y\neq x$ and let $\Gamma$ be the geodesic connecting $x$ to $y$, then, by assumption, $p_\Gamma(x_n)\rightarrow x$.
Then $d(x_n,y)\geq d(x_n,x)+o(1)$ 
and so the assertion follows from Proposition \ref{pr:Dconvergence} and the equivalence in SR spaces of $\Delta$ and polar convergence.

d)$\Rightarrow$ a)
Let $C$ be a compact set such that $x\in C$.
By compactness, if $(x^\prime_n)_{n\in\N}$ does not converge to $x$, exists $y\in E$, $y\neq x$ such that (a subsequence of) $(x^\prime_n)_{n\in\N}$ converges to $y$. 
Then, we deduce arguing as in the previous part,
that $d(x_n,x)\geq d(x_n,x^\prime_n)=d(x_n,y)+ o(1)$ in contradiction to Lemma \ref{lem:strong-dom}. 
\end{proof}

\begin{remark}
\label{re:Equivalence}
Note that the equivalence between a) and d), which makes sense in every metric space, holds for every SR space. Indeed, a) $\Rightarrow$ d) easily follows by applying a) for any given $y\in E$ by taking $C=\{x,y\}$.
\end{remark}

\begin{remark}
\label{re:Jost}
In \cite{Jost} the following properties of weak convergence are established: 
\begin{enumerate}
	\item Every bounded sequence has a weakly convergent subsequence (\cite[Theorem 2.1]{Jost})
	\item If $(x_n)_{n\in\N}$ is a bounded sequence, it has a renamed subsequence such that the corresponding sequence of 
``mean values'' or ``centers of gravity'' (see \cite[Definition 2.3]{Jost})	converges (see \cite[Theorem 2.2]{Jost}). 
	\item Consequently, every closed convex set is sequentially weakly closed and every lower semicontinuous convex function is weakly lower semicontinuous. 
	\item Every convex coercive function has a minimum (see \cite[Theorem 2.3]{Jost}).
\end{enumerate}
\end{remark}

Items (1), (3) and (4) also follow from the equivalence in Theorem \ref{te:Equivalence} and the following simple result. 
\begin{proposition}
\label{pr:CC->PC}
Let $(E,d)$ be a Hadamard space. Then every closed convex set $C\subset E$ is polarly closed, i.e. if $(x_n)_{n\in\N}\subset C$ and $x_n\rightharpoondown x$, then $x\in C$.
\end{proposition}
\begin{proof}
Let $I_{as}$ be defined as in (\ref{eq:Ias}). Then, by Lemma \ref{le:nonexpansive},
$I_{as}(p_C(x))=\limsup_n d(p_C(x),x_n)$ $=\limsup_n d(p_C(x),p_C(x_n))\leq \limsup_n d(x,x_n)=I_{as}(x)$, so $I_{as}(p_C(x))\leq I_{as}(x)$ i.e. $x\in C$.
\end{proof}

The question of topology associated with $\Delta$-convergence (or weak convergence) in Hadamard spaces was discussed by Monod in \cite{Monod} where he proposed another natural definition of weak convergence by
defining a topology $\mathcal T_C$ as the coarsest topology in which all closed convex sets are closed. We immediately see from (3) or from Proposition \ref{pr:CC->PC}
that weak convergence is finer than $\mathcal T_C$ convergence.
So, in view of property (1) above the two modes of convergence agree on bounded sequences whenever topology $\mathcal T_C$ is separated.
On the other side Monod suggests examples of some spaces 
(see lines between examples 20 and 21 in \cite{Monod}) on which $\mathcal T_C$ should be not separated. Therefore, in such a case, since the weak limit is obviously unique, the two modes of convergence would be different at least for nets.
In other words, whenever topology $\mathcal T_C$ gives a notion of limit different from the weak one, it has bad properties like nonuniqueness of the limit.

Moreover, on Hadamard spaces, Monod introduces another topology $\mathcal T_w$ which easily implies $\Delta$-convergence (and thus, weak convergence), but it is unclear how coarse is $\mathcal T_w$.  
$\mathcal T_w$ is defined (\cite[Par. 3.7]{Monod}) as the weakest topology under which for all $x,y\in E$ the map $z\mapsto d(x,z)-d(y,z)$ is continuous. In Hilbert spaces, if $z_k$ converges to zero in $\mathcal T_w$, then 
$\|x-z_k\|=\|x\|+\|z_k\|+o(1)$ for any $x$. Calculating squares of this relation one gets, assuming that $\|z_k\|$ is bounded away from zero, that 
$(x,z_k/\|z_k\|)\to \|x\|$, which, if is true for some $x$, is false for $-x$. This contradiction implies that $\mathcal T_w$-convergence in Hilbert space implies norm convergence, contrary to the statement in \cite[Example 18]{Monod}.

We refer the reader for further properties of weak convergence in Hadamard spaces, such as properties of Fejer monotone sequences, to the book of Bacak \cite{Bacak}.

Obviously the study of the notions of convergence introduced so far becomes interesting in metric spaces more general than Hadamard spaces where Opial condition, by analogy with $L^p$ spaces, does not hold.
For such spaces we believe that, in order to have a generalization of weak convergence, rather than $\Delta$-convergence, to metric spaces, definitions \ref{def:weak} and \ref{de:Jost} must be modified. Indeed, the proof of Theorem \ref{te:Equivalence} still applies in $L^p$, since the convex sets in $L^p$ spaces are SR (see also Remark \ref{re:Equivalence}). Therefore, all definitions of weak convergence given in this section are still equivalent, in the case of $L^p$, to polar convergence.
One may perhaps introduce a generalization of weak convergence in metric spaces that will be distinct from polar convergence, by replacing distance-minimizing projections in Definition~{de:Jost} with other maps, but addressing this issue is outside the scope of this note. When two distinct modes of convergence of weak type are defined, one may also consider a notion of ``double weak'' convergence in which one requires that a point is both a weak and a polar limit of the sequence (as happens, for instance, for $a.e.$ sequences in $L^p$ spaces, see Remark \ref{re:ae->w&p}). An application of this notion of double convergence, for which there is no compactness result of Banach-Alaoglu type, in $L^p$ spaces, is given at the end of the next section.

\section{Applications}
\label{se:app}

\subsection{Convergence of iterations to fixed points }
Fixed points of nonexpansive maps (maps with the Lipschitz constant
equal to 1) have been obtained as asymptotic centers of iterative
sequences (see \cite[Corollary to Theorem 2]{Edelstein} for suitable Banach spaces (see also \cite{C} and its review \cite{R}) and
\cite[Claim 2.7]{Staples} for suitable metric spaces).
Also in bounded complete SR metric spaces the existence of fixed points for a nonexpansive map $T$ easily follows, since, for any $x$, the asymptotic center of the iterations sequence  $(T^n(x))_{n\in\N}$ is a fixed point of 
$T$.
However, in general, it is not possible to get a fixed point as a polar limit of an iterations sequence. Indeed, Theorem \ref{thm:compactness} only gives the existence of the polar limit for a subsequence and the latter, as pointed out in the following example, (where the convergence is even in metric) is not necessarily a fixed point.

\begin{example}
Consider the map $T(x)=-x$ on a unit ball in any Banach space. Zero
is the unique fixed point of $T$ and the iterations sequence
$((-1)^{n}x)_{n\in\N}$ has zero as asymptotic center. On the
other hand, if $x\neq 0$, this sequence has no subsequence that $\Delta$-converges to zero and has two subsequences which converge to $x$ and $-x$, which are not fixed points.
\end{example}

On a bounded complete SR metric space, we can give the following characterization of polarly converging iterations sequences of a nonexpansive map.

\begin{proposition}
\label{pr:fixed}
Let $(E,d)$ be a bounded complete SR metric space. Let $x\in E$ and let $T:E\rightarrow E$ be a nonexpansive map.
Then, the following properties are equivalent.
\begin{itemize}
\item[(a)] The iterations sequence $(T^n(x))_{n\in\N}$ is polarly convergent;
\item[(b)] The polar limit of any subsequence of $(T^n(x))_{n\in\N}$  (when it exists) is a fixed point of $T$.
\end{itemize}
\end{proposition}

\begin{proof}
We shall only prove that (b) implies (a) (the converse implication being obvious by (iii) in Remark \ref{re:polarVsDelta}).
We shall prove, in particular, that 
if (b) holds true, then the sequence $(T^n(x))_{n\in\N}$ is maximal with respect to the relation $\preceq$ introduced in Remark \ref{re:ordering}.
To this purpose let $(T^{k_n}(x))_{n\in\N}$ be a subsequence of $(T^n(x))_{n\in\N}$ and let $r=\radas T^{k_n}(x)$.
By Theorem \ref{thm:compactness} we can replace $(T^{k_n}(x))_{n\in\N}$ by a subsequence (which we still denote by $(T^{k_n}(x))_{n\in\N}$, note that the value of $r$ does not increase)   
which is polarly converging to a point $y$.
Then, if $\varepsilon>0$ is an arbitrarily fixed real number,
there exists $\bar{n}\in\N$ such that 
$d(T^{k_{\bar{n}}}(x),y)<r+\varepsilon$. Since, by assumption (b), $y$ is a fixed point for 
$T$, we get, by induction, $d(T^n(x),y)<r+\varepsilon$, for any $n\geq k_{\bar{n}}$. 
(Indeed, if 
$n\geq k_{\bar{n}}$, we have
$d(T^{n+1}(x),y)=d(T^{n+1}(x),T(y))\leq d(T^n(x),y)<T^{k_{\bar{n}}}(x)<r+\varepsilon$, since $T$ is nonexpansive).
Therefore $\radas T^n(x) < r+\varepsilon$.
By the arbitrariness of $\varepsilon>0$ we get $\radas T^n(x)\leq r=\radas T^{k_n}(x)$.
\end{proof}

For spaces where the Opial condition holds, additional conditions,
such as asymptotic regularity, are imposed to assure weak convergence
of iterative sequences (see \cite{Browder},
\cite[Theorem 2]{Opial}, \cite{BR} and \cite{BBR}). 
However, the argument used in \cite{Opial} suggests 
that the role of Opial condition consists in deducing weak convergence
of iterations from their polar convergence. 

\begin{definition}
\label{de:PAR}
Let $(E,d)$ be a metric space, $T:E\rightarrow E$ and $x\in E$. We say that the map $T$ satisfies the 
PAR-condition (polar asymptotical regularity condition) at $x$ if, 
for any $y\in E$ polar limit of a subsequence $(T^{k_n}(x))_{n\in\N}$ of the iterations sequence $(T^n(x))_{n\in\N}$
and for all $\varepsilon>0$, there exists $\bar{n}\in\N$ such that
$d(T^{k_n-1}(x),y)< d(T^{k_n}(x),y) +\varepsilon$ for any $n\geq \bar{n}$.
\end{definition}

\begin{theorem}
\label{thm:polar2fix}
Let $(E,d)$ be a bounded complete SR metric space. Let $T:E\rightarrow E$ be a nonexpansive map which satisfies the PAR-condition at a point $x\in E$.
Then, the sequence $(T^n(x))_{n\in\N}$ polarly converges to a fixed point.
\end{theorem}

\begin{proof}
We shall prove (b) in Proposition \ref{pr:fixed}. Let $(T^{k_n}(x))_{n\in\N}$ be a polarly converging subsequence of $(T^n(x))_{n\in\N}$ and let $y$ be the polar limit of 
$(T^{k_n}(x))_{n\in\N}$. Since $T$ is a nonexpansive map which satisfies the PAR-condition in $x\in E$ we have that 
$d(T^{k_n}(x),T(y))\leq d(T^{k_n-1}(x),y)< d(T^{k_n}(x),y)+\varepsilon$ for any $\varepsilon>0$ for large $n$. Therefore also $T(y)$ is an asymptotic center of $(T^{k_n}(x))_{n\in\N}$ and, since in SR spaces the asymptotic center is unique, we get $T(y)=y$.
\end{proof}

\subsection{Brezis-Lieb Lemma without $a.e.$ convergence}
\hspace{0.5 cm}
The celebrated Brezis-Lieb Lemma (\cite{Brezis-Lieb}) is stated for (bounded) $a.e.$ converging
sequences that, see Remark \ref{re:ae->w&p}, are also polarly converging. 
Remarkably, convergence
$a.e.$ is not needed when $p=2$, which suggests that some version of
Brezis-Lieb Lemma may hold for other $p$ if $a.e.$ convergence in the
assumption is replaced by the double weak convergence mentioned in the previous section (weak and polar convergence of $L^p$). Indeed, in \cite{ST} the following result is proved.
\begin{theorem}
\label{thm:newbl} Let $(\Omega,\mu)$ be a measure
space. Assume that $u_{n}\rightharpoonup u$ and $u_{n}\rightharpoondown u$
in $L^{p}(\Omega,\mu)$. If $p\ge3$ then

\begin{equation}
\int_{\Omega}|u_{n}|^{p}d\mu\ge
\int_{\Omega}|u|^{p}d\mu+\int_{\Omega}|u_{n}-u|^{p}d\mu+o(1).\label{eq:BL}
\end{equation}
\end{theorem}

It is shown in \cite{AT} that the condition $p\ge3$ cannot be removed,
except when $p=2$. Moreover, one can see by easy examples with $p=4$ (see \cite{ST}) that, in general, the equality does not hold in (\ref{eq:BL}).

The elementary relation for Hilbert spaces, 
$$u_n\rightharpoonup u\;\Longrightarrow\; \|u_k\|^2=\|u\|^2+\|u_k-u\|^2+o(1),$$
has the following counterpart in $\mathrm{CAT}(0)$-spaces.

\begin{theorem}
\label{thm:newBLpL} Let $(E,d)$ be a space which satisfies inequality (\ref{cat}) and property (\ref{eq:d=2r}) for all $x,y\in E$.
Let $(x_n)_{n\in N}\subset E$, $x\in E$ such that 
$x_{n}\rightharpoondown x$.
Then, for all $y\in E$
\begin{equation}
\label{eq:BLPolar}
d^2(x_n,y)\geq d^2(x,y)+ d^2(x_n,x)+o(1).
\end{equation}
\end{theorem}

\begin{proof}
Let $y\in E$ and given a dyadic number $\alpha\in ]0,1[$ let $c_\alpha$ as in (\ref{eq:c_alpha}).
For $n$ large 
\begin{equation}
\label{eq:polarineq}
d(x_n,x)< d(x_n,c_\alpha).
\end{equation}
So, by taking $z=x_n$ in (\ref{catalpha})  we deduce 
$$
d^2(x_n,x)\leq d^2(x_n,y)-(1-\alpha) d^2(x,y).
$$
Since we can take $\alpha$ arbitrarily close to $0$, the assertion follows.
\end{proof}

From Theorem \ref{thm:newBLpL} one can easily deduce some classical properties of weak convergence known for normed spaces.
The first one is just another form of item (3) in Remark \ref{re:Jost}.

\begin{corollary}
\label{cor:BLpL1}
Let $(E,d)$ be a Hadamard space, then
for any $y\in E$ the map $x\in E\mapsto d(x,y)\in \R$ is a polarly lower semicontinuous function.
\end{corollary}

\begin{corollary}
\label{cor:BLpL2}
Hadamard spaces possess the Kadec-Klee property,
i.e.
if $(x_n)_{n\in\N}\subset E$ and $x\in E$ are such that $x_n\rightharpoondown x$ and if there exists $y\in E$ such that 
$d(x_n,y)\rightarrow d(x,y)$
then
$d(x_n,x)\rightarrow 0$.
\end{corollary}

\noindent{Acknowledgment.} The authors thank the reviewer of this note for providing a broader perspective to their results, including the crucial references to the work of Lim and Kirk \& Panayak cited below. One of the authors (C.T.) thanks the math faculties of the Politecnico and of the University of Bari for their warm hospitality.

\bibliographystyle{amsplain}

\end{document}